\documentclass[preprint,12pt]{article}
\usepackage{amsmath,amsthm,amssymb,latexsym}
\usepackage{doc}
\usepackage{color}
\usepackage{makeidx}
\usepackage{hyperref}
\usepackage{url}
\usepackage[dvips]{graphicx}

\newtheorem{thm}{Theorem}[section]
\newtheorem{cor}[thm]{Corollary}
\newtheorem{lemma}[thm]{Lemma}
\newtheorem{prop}[thm]{Proposition}

\makeindex

\begin{document}

\begin{center}
{\LARGE{Contrast in Greyscales of Graphs}} \vspace*{.5cm}

{\sc  Natalia de Castro\footnote[1]{Dpto. de Matem\'atica Aplicada I, Universidad de Sevilla, Avda.
Reina Mercedes s/n, Sevilla, SPAIN, e-mail: natalia@us.es, vizuete@us.es, rafarob@us.es }
 Mar\'{\i}a A. Garrido-Vizuete$^1$, Rafael Robles$^1$,
  Mar\'{\i}a Trinidad Villar-Li\~n\'an\footnote[2]{Dpto. de Geometr\'{\i}a y Topolog\'{\i}a, Facultad de Matem\'aticas,
Universidad de Sevilla, C/ Tarfia s/n 41012, Sevilla,
SPAIN, villar@us.es} \\
}

\end{center}

\graphicspath{{Figures/}}

\begin{abstract}
\noindent In this work we present the notion of greyscale of a graph as a colouring of its vertices that uses colours from the real interval [0,1]. Any greyscale induces another colouring by assigning to each edge the non-negative difference between the colours of its vertices. These edge colours are ordered in lexicographical increasing  ordering and gives rise to a new element of the graph: the contrast vector. We introduce the notion of maximum contrast vector (in the set of contrast vectors of all possible greyscales defined on the graph) as a new invariant for the graph. The relation between finding the maximum contrast vector for the graph and its chromatic number is established. Thus the maximum contrast problem is an NP-complete problem. However, the set of values of any maximum contrast greyscale for any graph is bounded by a finite set which is given. Several methods to compute the maximum contrast vector with some restrictions  are collected in this paper.

The interest of these new concepts lies in their possible applications for solving problems of engineering, physics and applied mathematics which are modeled according to a network whose nodes have assigned numerical values of a certain parameter delimited by a range of real numbers. The objective is to maximize the differences between each node and its neighbors, from a local and global point of view simultaneously  through a vectorial objective function, that is the contrast vector.

% A {\it greyscale} $f$ of a graph $G(V,E)$ is a mapping from $V$ to the interval $[0,1]$ such
%that $\{0, 1\} \subseteq Im(f)$. This function $f$ induces another
%mapping $\widehat{f}$ on $E$ by assigning to each edge the non-negative difference of %the values of
%$f$ on its vertices. The {\it contrast vector} $cont(G,f)$ is defined as the vector %$(\widehat{f}(e_1),
%\widehat{f}(e_2), \ldots, \widehat{f}(e_m))$ for all edges $e_i$ of $G$ in such a way that
%$\widehat{f}(e_i) \leq \widehat{f}(e_{i+1})$  for $i= 1, 2, \dots, m-1$. The concept of %{\it  maximum
%contrast vector} is  presented by using the lexicographical ordering in the set of %contrast vectors
%of all possible greyscales defined on $G$ and a greyscale that gives rise to  a maximum %contrast vector is named {\it maximum contrast greyscale}.

%The relation between finding the maximum contrast vector for the graph
%$G$ and the chromatic number of $G$ is established.
 %Thus the maximum contrast problem is an NP-complete problem.  
%However, the set of values of any maximum contrast greyscale  for any graph is bounded %by  a finite set which is given. Several 
%methods to compute the maximum contrast vector with some restrictions in 
 %are collected in this paper.
\end{abstract}

\vspace*{0.4cm}

\vspace*{0.4cm}
\textbf{MSC 2010 (Primary):} 05C, 68R.    \textbf{MSC 2010 (Secondary):} 05C15, 05C85,  90C47.

\section{Introduction}

Graph colouring is one of the most studied problems of combinatorial optimization because it has a wide variety of applications such as wiring of printed circuits~\cite{clw-odmwi-92}, resource allocation~\cite{wsn-radpbnugca-91}, frequency assignment problem~\cite{ahkms-mp-07, gk-glvws-09, omgh-bsgc-16}, a wide variety of scheduling problems~\cite{m-gcpas-04} or computer
register allocation~\cite{cacchm-rac-81}.\par

A variety of combinatorial optimization problems on graphs can be formulated  similarly in the
following way. Given a graph $G(V,E)$, a mapping $f:V\longrightarrow \mathbb{Z}$ is defined
and it induces a new mapping $\widehat f:E\longrightarrow \mathbb{Z}$ by $\widehat
f(e)=|f(u)-f(v)|$ for every $e=\{u,v\}\in E$. Then an optimization problem is formulated from
several key elements: mappings $f$ belonging to a  subset $S$, the image of $V$ by $f$
and the image of $E$ by $\widehat f$. In particular, the classic graph colouring problem, that is,
colouring the vertices of $G$ with as few colours as possible so that adjacent vertices always have
different colours, can be stated in these terms as follows:
$$\chi(G)=\min_{f\in S} |f(V)|\ \ \text{where}\ \ S=\{ f:V\rightarrow \mathbb{Z} \text{ such that } 0\notin \widehat f(E)\}. $$
It is well known that this minimum number $\chi(G)$ is called the chromatic  number of the graph
$G$ and that its computing is an NP-hard problem~\cite{k-rcp-72}.

It must be stressed that the classic graph colouring problem takes into  account the number of
colours used but not what they are. However, there are several works related to map colouring for
which the nature of the colours is essential, whereas the number of them is fixed. The {\it maximum
differential graph colouring problem} ~\cite{hkv-mdgc-11}, or equivalently the {\it antibandwidth
problem}~\cite{lvw-svbmp-84}, colours the vertices of the graph in order to maximize the smallest
colour difference between adjacent vertices and using  all the colours $1,2,\ldots,|V|$. Under the
above formulation, these problems are posed as follows:
$$\max_{f\in S}\min \widehat f(E)\ \ \text{ for } S=\{ f:V\rightarrow \mathbb{Z} \text{ such that }  f(V)=\{1,2,\dots,|V|\}\},$$
and therefore the complementary optimization case, the {\it bandwidth problem}, is given by
$$\min_{f\in S}\max \widehat f(E)\ \ \text{ for } S=\{ f:V\rightarrow \mathbb{Z} \text{ such that }  f(V)=\{1,2,\dots,|V|\}\}.$$
Note that these problems are concerned with the optimization of a scalar function and mappings that
take values within a discrete spectrum. Dillencourt et al.~\cite{deg-ccggcse-07} studied a variation of
the differential graph colouring problem under the assumption that all colours in the colour
spectrum are available, more precisely, the space of colours is a three dimensional subset of
$\mathbb{R}^3$. This makes the problem continuous rather than discrete since the mapping $f$ has
image in $\mathbb{R}^3$ (see ~\cite{deg-ccggcse-07} for details).

\par

Other colouring problems are included in  {\it the examination scheduling problem} category (see for instance
\cite{b-esmmsst-98}). This problem consists of
assigning a number of exams to a number of potential time  periods or slots within the examination
period  taking into account that no student can take  two or more exams at the same time.
The graph $G$ associated to
the examination scheduling problem has a vertex for each exam and two vertices are adjacent
whenever there is at least one  student taking their corresponding exams. This way, the chromatic
number of $G$ provides the minimum number of slots needed to generate an examination period
schedule.

When a weight $w_i$ is associated to  each  colour $i$ in a proper colouring of $G$, and the sum of that colour weights is minimized,  the  optimization problem is known to be the {\it minimum sum colouring problem} whose applications to  scheduling  problems and resource allocations are recently  developed (see \cite{llc-mscpubcs-17} and the references therein).

In all these approaches the colouring functions are considered to be  scalar mappings.
In this work, we present an alternative method that can solve colouring problems but focussing in the optimization of
a vector that measures the differences of colours between adjacent vertices in the graph.

Thus, different concepts related to the colouring of a graph $G=G(V,E)$ are accurately introduced
in Section \ref{sec:prelimi}: the {\it contrast} and the {\it gradation} associated to a greyscale
of $G$. Namely, given the graph $G$, a greyscale is a mapping that associates a value  from the interval $[0,1]$ to each vertex
$v\in V$. This assignment can be understood as an extension of the colouring of the vertices of $G$
with  grey tones.  For the  contrast problem, the objective is to maximize the minimum difference of tones of grey  between extremes  of any edge.  
For the gradation problem, the goal is to minimize the maximum
difference of tones of grey between extremes of any edge. 
The  gradation  problem is widely studied in a work by the same authors of this paper
\cite{cas-ga-ri-ro-vi-14}.

By  introducing this new problem  based in a vectorial function which  allocates grey tones in an equitable manner, we propose another way to  maximize the difference of colours 
 in the graph. More specifically, we are not
interested in maximize the total amount of ``contrast'' (here contrast means the difference of colours between adjacent vertices) neither maximize the minimum contrast (scalar
function) but maximize the vector whose components represent  local contrast of adjacent vertices,
in ascending order. This way, the main advantage of our proposal
lies in the possibility  of obtaining a local distribution of the
contrast for every vertex in the graph.

\begin{figure}[htb]
 \begin{center}
\includegraphics[width=7cm]{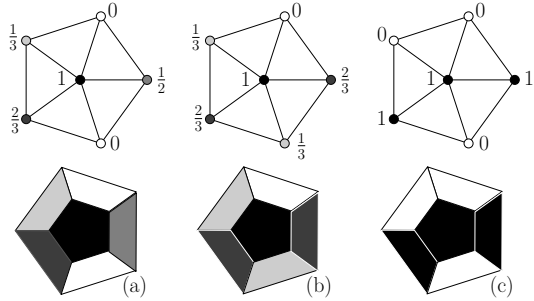}
\caption{A comparison of three greyscales for the wheel according to  their contrasts.}
\label{fig:veccontrast}
\end{center}
\end{figure}

Figure \ref{fig:veccontrast}  visually shows the goodness of our vectorial optimization  versus a scalar optimization,  based on  total amount of contrast on edges or alternatively  on maximin criterium.

Let us illustrate this fact by considering  three colourings of the  faces of the map in  Figure
\ref{fig:veccontrast}. We construct its dual graph  avoiding the external face in such a way that the resulting graph is the wheel. 
Three greyscales for  the wheel are presented. 
Observe that according to the total amount of contrast criterium, Figure \ref{fig:veccontrast} (c) is an optimal solution (with total amount of contrast equal to 7  versus 35/6 and 5 in Figure \ref{fig:veccontrast} (a) and (b), respectively).  However, in (c) there are some adjacent faces having no contrast.
On the other hand, (a) and (b) are solutions under  the scalar maximin criterium (both of them have  minimum contrast on edges  equal to $1/3$). Nevertheless, only (a) is an optimal  solution under our vectorial criterium.  In (a), it is not difficult to check  that  every  pair of adjacent faces has the maximum possible  contrast.

\par

\textit{Outline of the paper}. This paper is organized as follows. In Section 2  the necessary definitions about the contrast problem on graphs are established.
Section 3 is devoted to the maximum  contrast problem  NP-completeness nature. This  property  is deduced from the narrow relation between the
classical colouring problem and the contrast problem. Nevertheless, the set of possible
values for a maximum  contrast greyscale can be bounded by a set determined algorithmically as it is  collected in Section \ref{sec:values}. Also,  an algorithm that calculates this set of possible values of a maximum contrast greyscale  for any graph is presented and they are obtained for graphs with chromatic number up to 7.  We define also restricted versions of the original problem when the grey tones 0 or 1 of some
vertices are a priori known and the aim is to obtain the maximum contrast vector preserving such fixed grey tones. This problem is studied in Section 5 and it is solved for the family  of complete bipartite graphs. It is also analysed in  some other particular cases such as bipartite graphs and trees with some additional conditions over the set of  vertices initially coloured. The last section contains a brief review of open problems and future works.

\section{Preliminaries}
\label{sec:prelimi}

This section is devoted to establishing the basic concepts about contrast on graphs and to formulating
the problems to be studied in this paper. All over this paper, a graph  is finite,
undirected and simple and is denoted by $G(V,E)$, where $V$ and $E$ are its vertex-set and
edge-set, respectively.  The number of elements of $V$ and $E$ are denoted by $n$ and $m$, respectively.
 Let  $N(v)$ denote the set of neighbours of the vertex $v$ and let $deg(v)$ denote the {\it degree} of $v$, that is the cardinal of 
$N(v)$. For further  terminology we follow~\cite{h-gt-90}.

 Given a graph $G(V,E)$, a \textit{greyscale} $f$ of $G$
is a mapping on $V$ to the interval $[0,1]$ such that  values 0 and 1 belong to $Im(f)$.
 For each vertex $v$ of $G$, we call $f(v)$ the
\textit{grey tone} of $v$, or more generally, the \textit{colour} of $v$. Notice that two adjacent
vertices may have mapped the same grey tone. In particular,  values $0$ and $1$ are called
\textit{the extreme tones}  or \textit{white and black colours}, respectively.
 In a natural way, the notion of
\textit{complementary greyscale} arises for each
greyscale $f$ such that it maps every vertex $v$ of $G$ to $1-f(v)$.

Associated to each greyscale $f$ of the graph $G(V,E)$,
 the mapping $\widehat{f}$ is defined on $E$ to the interval $[0,1]$ as $\widehat{f}(e)=|f(u)-f(v)|$ for each $e=\{u,v\}\in E$
 and it represents the gap or increase between the grey tones of  vertices $u$ and $v$.
 The value $\widehat{f}(e)$ is also said to be the \textit{grey tone} of the edge $e$. Thus, we deal with
 \textit{coloured vertices} and \textit{edges} by $f$ and  $\widehat{f}$, respectively. Notice that, for any greyscale $f$, the same mapping
  $\widehat{f}$ associated to  $f$ and its complementary one are obtained.

The \textit{contrast vector}  
associated to the greyscale $f$ of $G$ is defined to be the vector $cont(G,f)=(\widehat{f}(e_1), \widehat{f}(e_2),
 \ldots, \widehat{f}(e_m))$      
 where the edges of $G$ are indexed in such a way that $\widehat{f}(e_i) \leq \widehat{f}(e_j)$ for $i < j$,
  that is, in ascending order of their grey tones. 
    For the sake of clarity
  and when the graph is fixed, it can be denoted $cont(G,f)={\mathcal{C}}_f.$ 
   Figure~\ref{fig:greyscaleK4} shows two
greyscales of the graph $K_4$, $f$ and $f'$,  whose corresponding  contrast vectors are  ${\mathcal{C}}_f=(0, \frac{1}{2}, \frac{1}{2},  \frac{1}{2}, \frac{1}{2},1)$ and ${\mathcal{C}}_{f'}=(\frac{1}{3}, \frac{1}{3}, \frac{1}{3},  \frac{2}{3}, \frac{2}{3},1),$ respectively.

Given two greyscales $f$ and $f'$ of a graph $G$, we say that $f$ has \textit{better contrast} than
$f'$ if their corresponding contrast vectors verify  ${\mathcal{C}}_f>{\mathcal{C}}_{f'}$,
following the lexicographical order. Thus, the ascending order of contrast vectors determines the
goodness in terms of contrast. Then, $f$ is said to be \textit{smaller or greater by contrast} than
$f'$ if ${\mathcal{C}}_f < {\mathcal{C}}_{f'}$ or ${\mathcal{C}}_f > {\mathcal{C}}_{f'}$,
respectively. In Figure~\ref{fig:greyscaleK4}, the greyscale $f'$ has better contrast than $f$.
The {\it  maximum contrast vector}  is  defined as $cont_{max}(G)= max \{cont(G,f)\, \mbox{such that} \,  f  \; \mbox{is a greyscale of }\, G\}$.
If $f$ is a greyscale of $G$ which gives rise to the vector $cont_{max}(G)$, we will say that $f$
is a {\it maximum contrast greyscale} of $G$ and the first component of $cont_{max}(G)$
will be called the {\it lightest tone} of $G$ and denoted $lt_G$.

In a similar way,  the \textit{gradation vector} associated to the greyscale $f$ of $G$ is the
vector defined as $grad(G,f) = {\mathcal{G}}_f = (\widehat{f}(e_m), \widehat{f}(e_{m-1}) , \ldots ,
\widehat{f}(e_1))$. The components of any gradation vector are ordered in decreasing order and  $f$
has \textit{better gradation} than $f'$ if their corresponding gradation vectors verify
${\mathcal{G}}_f<{\mathcal{G}}_{f'}$, according to the lexicographical order. In Figure~\ref{fig:greyscaleK4}, ${\mathcal{G}}_f= (1, \frac{1}{2}, \frac{1}{2},  \frac{1}{2}, \frac{1}{2},0) < {\mathcal{G}}_{f'}= (1, \frac{2}{3}, \frac{2}{3},  \frac{1}{3}, \frac{1}{3},\frac{1}{3})$, hence the greyscale $f$ has better gradation than $f'$.
 See \cite{cas-ga-ri-ro-vi-14} for more details.

\begin{figure}
\begin{center}
\includegraphics[width=9cm]{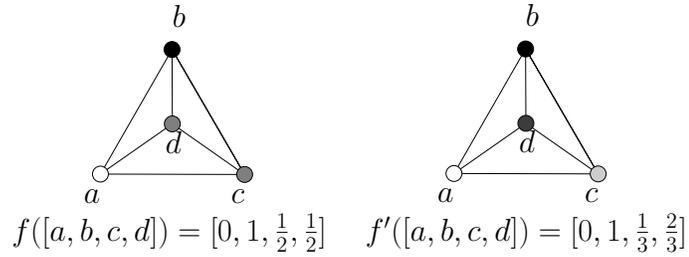}
\caption{Two greyscales $f$ and $f'$  of the graph $K_4$. } \label{fig:greyscaleK4}
\end{center}
\end{figure}

The following problem arises naturally in the context of contrast. It is posed
for connected graphs but general graphs can be also considered and analogous results hold when
working with each one of their connected components.

\vspace*{.3cm}
 \textbf{Maximum
contrast on graphs} (\textsc{macg})

\textit{Instance:} Connected graph $G(V,E)$.

\textit{Question:} Is it possible to find a greyscale of $G$ such
that its contrast vector is maximum?

\vspace*{.3cm} We deal with the restricted version of this problem, namely 
when the grey tones of some vertices are a priori known and the aim
is to obtain the maximum contrast vector preserving these fixed grey
tones.   We  focuss on a particular case of this  problem, namely when only  white and black tones are fixed. 

Given a graph $G(V,E)$ and a nonempty proper subset $V_c$ of $V,$
\textit{an incomplete $V_c$-greyscale} of $G$ is a mapping on $V_c$
to the interval $[0,1]$. The vertices of $V_c$ are named {\it initially coloured vertices}. A greyscale $f$ is
\textit{compatible} with an incomplete $V_c$-greyscale $g$ if
$f(u)=g(u)$ for all $u \in V_c$. The process of obtaining such an $f$ is
called {\it extending the incomplete greyscale $g$}.
 The problem is established as follows and it will be formally studied in 
  Section \ref{sec:restricted}.

\vspace*{.3cm}

 \textbf{$\{0,1\}$-Restricted maximum contrast on graphs}
($\{0,1\}$-\textsc{rmacg})

\vspace*{.3cm}
\textit{Instance:} Connected graph $G(V,E)$ and  an incomplete $V_c$-greyscale $g$, where $V_c=V_0\cup V_1 \subset V$ with $V_0$ and $V_1$
disjoint subsets  and such that $g(v)=0$ for
$v \in V_0$ and $g(v)=1$ for $v \in V_1$.

\textit{Question:} Is it possible to find a greyscale $f$ of $G$
compatible with $g$ such that its contrast vector is maximum in the
set of contrast vectors of all possible  greyscales compatible with
$g$ defined on $G$?

If a greyscale $f$ provides an affirmative answer to the above problem we say that $f$ is a  {\it maximum contrast greyscale for the $\{0,1\}$-\textsc{rmacg} problem}. Note  that  $f$ is a maximum contrast greyscale compatible with the greyscale given in  the instance.  The contrast vector associated to 
$f$ is named {\it the maximum contrast vector for the $\{0,1\}$-\textsc{rmacg} problem}.

  \section{Maximum contrast problem}\label{sec:problem}
\label{sec:contrast} In this section the problem of finding out  the maximum contrast of a graph,
denoted \textsc{macg},  is studied. Given a graph $G(V,E)$, this problem consists on knowing
whether   a greyscale  whose contrast vector is maximum can be found for $G$. A relation between the chromatic number  and the lightest tone of the graph $G$ is shown. The main consequence of this property is the NP-completeness of the problem {\sc macg}.

  Let $G(V,E)$ be a connected graph and let $f$ be a greyscale of $G$. 
   Our purpose is to obtain the maximum contrast vector, namely $cont_{max}(G) = max \{cont(G,f)\,$ $ \mbox{such that} \,  f \; \mbox{is a greyscale of }\, G\}$. Therefore,
answering the question proposed in {\sc macg}  can be considered as a
maximin problem. According to the definition of better
contrast, given in Section \ref{sec:prelimi}, it is clear that the
maximum contrast vector $cont_{max}(G)$ has no component equal to 0.
 Also, it is  immediately deduced that a necessary condition for a greyscale $f$ to
 be a maximum contrast greyscale of $G$
 is that for any vertex $v$ with degree 1, $f(v)$ is an extreme tone.

Given  a  greyscale $f$ of $G$,  
 an \emph{incremental path of length} $k
\in N$ for $f$ is defined to be a path  of $G,$  
$P_k=\{u_0,e_1,u_1,e_2,u_2,\dots ,e_k,u_k \}$ with $f(u_i)=\frac{i}{k}$ for $i=0,1,\dots ,k$. Thus,
in any incremental path of length $k$ all edges are coloured with the grey tone $\frac{1}{k}$.

We demonstrate that if $f$ is a maximum contrast greyscale of $G$, then the edges coloured with the lightest tone are located in some incremental path.

Next, we study the relation between the chromatic number and the vector
$cont_{max}(G)$. First,  let us consider a  graph $G$ with chromatic number
$\chi(G)=2$. Any 2-colouring $f$ of $G$ provides a greyscale with
colours 0 and 1 and a contrast vector with  all components equal to 1, hence  this greyscale gives rise to the vector
$cont_{max}(G)$. We observe that in
this particular case,  the chromatic
number equals  the number of grey tones in any maximum contrast greyscale of 
 $G$.  This property is far from the more general case of connected graphs with $\chi(G)\geq 3$, which is studied  next.

\begin{lemma}\label{previo}
Let  $G(V,E)$ be a connected graph and let $f$ be a maximum contrast
greyscale of $G$.
 Let $v\in V$ be a vertex  such that   $0<f(v)<
1$, then there exist  $u_1$ and $u_2\in N(v)$  satisfying both
following assertions:
\begin{enumerate}
\item $f(u_1)<f(v)<f(u_2).$
\item $\widehat{f}(\{u_1,v\}) = \widehat{f}(\{u_2,v\}) =min \{\widehat{f}(\{u,v\}) : u\in N(v)\}.$
\end{enumerate}
\end{lemma}

\begin{proof} Let us consider
$a= \min \{ \widehat{f}(\{u,v\}): u\in N(v) \} $. Since $f$ is a maximum contrast greyscale, then $a> 0$ trivialy and
since $f(v)<1$, it is clear that $a<1$.\par 
Let us define the set $\mathcal{A} = \{ e \in E : \
\widehat{f}(e)=a \}$, then the maximum contrast vector  has  $|\mathcal{A}|$ coordinates equal to $a$, namely 
$C_{f}=(c_1,c_2,\dots ,c_r, a,a,\dots ,a, \dots ).$ In  case that  $a$ is
equal to the lightest tone of $G$ then $r=0$.
\par Let us now
partition the set $N(v)=M_v \cup N_v$, as follows: 
$$M_v= \{ u \in N(v) : \widehat{f}(\{u,v\})=a \},\quad
N_v= \{ u \in N(v) : \widehat{f}(\{u,v\})>a \}.$$

Let us define  $b=\min \{ \widehat{f}(\{u,v\}) : u \in
N_v \}$ if $N_v \neq \emptyset$ and $b=1$  otherwise. Since $a<b$,
let us consider  $\varepsilon>0$, such that $a+2 \varepsilon < b$
and $0<f(v)-\varepsilon < f(v)+ \varepsilon < 1$. Two new
greyscales $f_{v^+}$ and $f_{v^-}$ are now defined as follows: \vskip.5cm

$f_{v^+}(w)=\left\{\begin{array}{ccc}
                    f(w) & \mbox{if} \ \ \  w\neq v  &  \  \\
                     & & \\
                    f(v)+\varepsilon & \mbox{if} \ \ \  w=v  &
                   \end{array}
                  \right.$  
 $f_{v^-}(w)=\left\{\begin{array}{ccc}
                    f(w) & \mbox{if} \ \ \  w\neq v  & \  \\
                     & & \\
                    f(v)-\varepsilon & \mbox{if} \ \ \   w=v  &
                   \end{array}
                  \right.$ \vskip.5cm
By definition, the greyscales  $f_{v^\pm}$ verify
$\widehat{f}_{v^\pm}(\{w_1,w_2\})=\widehat{f}(\{w_1,w_2\})$
for all edges $\{w_1,w_2\}$ of $G$ with $w_1,w_2 \neq v$ and
$\widehat{f}_{v^\pm}(\{w_1,v\})= a \pm \varepsilon < b\pm \varepsilon =\widehat{f}_{v^\pm}(\{w_2,v\})$ for all
edges $\{w_1,v\}$, $\{w_2,v\}$ of $G$ with $w_1 \in M_v$ and $w_2 \in
N_v$, in case that $N_v\neq \emptyset$.\par

Let us  prove statement  1 by reductio ad adsurdum. By assuming  $f(u)\geq f(v)$ for all $u \in M_v$, we obtain $\widehat{f}_{v^-}(\{u,v\})=a+\varepsilon$.
In an analogous way, if we suppose  $f(u)\leq f(v)$ for all $u \in M_v$, then
$\widehat{f}_{v^+}(\{u,v\})=a+\varepsilon$. Now, by using the properties of the functions
$f_{v^+}$ and $f_{v^-}$, it is readily deduced that both  contrast vectors associated to these
greyscales verify
$$C_{f_{v^\pm}}=(c_1,c_2,\dots ,c_r, a, \dots ,a,a+\varepsilon, \dots, a+\varepsilon, \dots)$$
with $|\mathcal{A}|-|M_v|$ coordinates equal to $a$ and at least $|M_v|$
coordinates equal to $a+\varepsilon$. Obviously, the contrast
vectors associated to $f_{v^+}$ and $f_{v^-}$ are better than the
contrast vector associated to $f$, contradicting that $f$ is a
maximum contrast greyscale. Consequently, both  assumptions $f(v)\leq f(u)$ and $f(u)\leq f(v)$  for all $u \in M_v$  are false
and there is at least one vertex $u_1 \in M_v$ and one vertex $u_2
\in M_v$ such that $f(u_1)<f(v)<f(u_2)$. Since $u_1$ and $u_2$ belongs to $M_v$, then $\widehat{f}(\{u_1,v\})=\widehat{f}(\{u_2,v\})=a$ and statement 2 holds.
\end{proof}

 From now on,  the pair of vertices $u_1$ and  $u_2$ associated to  a vertex $v$, given by Lemma \ref{previo}, will be named {\it  pair of neighbours closest to
$v$}, the vertex $u_1$ will be named {\it the neighbour closest to $v$ on
the left} and the vertex $u_2$ will be named {\it the neighbour closest
to $v$ on the right}.

Let us remark  that the existence of such pair of vertices $u_1, u_2$ will
be determinant in the searching of possible values of
$Im(f)$, for any maximum  contrast greyscale $f$ of $G$. The following proposition  shows a first result concerning $Im(f)$  although this set   will be studied in detail in Section \ref{sec:values}.

\begin{prop}\label{previo2}
Let $G(V,E)$ be a connected graph   and let $f$ be a maximum contrast greyscale of $G$. If there exists a vertex $v\in V$ such that $0<f(v)<1$, then $lt_G\leq \frac12$. Moreover, if  $lt_G= \frac12$, then  $Im(f) =  \{0,\,  \frac{1}{2}, \, 1\}$.
\end{prop}
\begin{proof}
	
	If $f(v)$ is not $0$ nor  $1$  then, from Lemma \ref{previo}, there exists a pair of neighbours closest to
	$v$, say $u_1$ and $u_2$, with
	$0\leq f(u_1)<f(v)<f(u_2)\leq 1$ such that $f(v)-f(u_1)=f(u_2)-f(v)=\min_{u\in N(v)} \hat{f}(\{u,v\})\ge lt_G$.
	Therefore,  $1\ge f(u_2)-f(u_1)\ge 2 lt_G$ is held and then,  $lt_G\le\frac12$. In the particular case that $lt_G=\frac 12$ we obtain  $f(u_2)-f(u_1)=1$, which provides necessarily  $f(u_1)=0, f(u_2)=1$ and $f(v)=\frac 12$. Consequently, we reach	$Im(f) = \{0,\,  \frac{1}{2}, \, 1\}$.
	
\end{proof}

The next lemma shows that the lightest tone  $lt_G$  is a rational number $\frac{1}{k}$,
being $k$ a natural number. Moreover, the lightest tone is the colour of the edges of  any
incremental path of length $k$.

\begin{lemma}\label{pricamino}
 Let  $G(V,E)$ be a connected graph and let $f$ be a maximum
contrast greyscale of $G$ with lightest tone  $lt_G$. The following statements hold:
\begin{enumerate}
   \item There exists a natural number $k$ such that $lt_G=\frac{1}{k}.$
   \item For every  $e \in E$ with $\widehat{f}(e)=lt_G$ there is at least one incremental path of length $k$ containing
   $e$.
   \item The vector $cont_{max}(G)$ has  at least $k$ components equal to $lt_G$.
   \item The set $I_k= \{\frac{i}{k} \ : \  i=0, \, 1,\dots , k \}$ is a subset of $Im(f)$.
 \end{enumerate}

\end{lemma}

\begin{proof}

Without loss of generality, we can suppose  $lt_G<1$, since otherwise the result is obviously true for
$k=1$: the incremental path of length $1$ is precisely the edge $e$.

Let us consider an edge $e=\{u_0,v_0\}$ such that
$\widehat{f}(e)=lt_G<1$. In this case, it is clear that $\{f(u_0),
f(v_0)\} \neq \{0, 1\}$ and at least one of these  vertices has degree at least two since $G$ is connected. We suppose that $f(u_0)<f(v_0)$, otherwise
the proof is similar.

Firstly, let us consider the case $0<f(u_0)<f(v_0)<1$,  Lemma \ref{previo}  ensures the existence of a neighbour $v_1$ closest to $v_0$ on the right satisfying $\widehat{f}(\{v_0,v_1\})=lt_G$.
Hence, by applying   Lemma \ref{previo} to the vertex $v_1$, it is possible to obtain a new
neighbour $v_2$ closest to $v_1$ on the right such that $\widehat{f}(\{v_1,v_2\})=lt_G$ and by iterating this process, a sequence of
 neighbours closest on the right
 $\{v_0,v_1,\dots ,v_{s} \}$ such that $f(v_s)=1$ and $s=\displaystyle{\left[ \frac{1-f(v_0)}{lt_G} \right]}$ is constructed.
Analogous reasoning is applied on the left of $u_0$. In this situation,  Lemma \ref{previo}  ensures the existence of a
neighbour $u_1$ closest to $u_0$ on the left. 
 Hence, by iterating this procedure, it is possible to obtain a sequence of neighbours closest on the left
  $\{u_{r},u_{r-1},\dots ,u_0 \}$ such that
 $f(u_r)=0$ and $r=\displaystyle{\left[ \frac{f(u_o)}{lt_G} \right]}$.

Let $P_k=\{u_r,\dots ,u_0,v_0,v_1,\dots ,v_s \}$, with $k=r+s+1$, stand for the
incremental path of $G$. Note that $P_k$ satisfies
 $0=f(u_r)<f(u_{i-1})< \dots <f(u_0)<f(v_0)< \dots f(v_s)=1$.
 
 Secondly, let us consider the case of $f(u_0)=0$, hence the path $P_k$ starts with $u_0$, that is $r=0$. Analougously, in case that $f(v_0)=1$ the path $P_k$ finishes with $v_0$,
  that is $s=0$.
 
Finally, from the above construction, the interval $[0,1]$ is divided into
$k$  subintervals of equal length $lt_G$. Hence $lt_G=\frac{1}{k}$ with $k
\in \mathbb{N}$. This shows assertions 1 and 2. Observe that the $ith$-vertex 
 of $P_k$ takes the grey tone  $\frac{i}{k}$ (for $i=0, 1, \dots, k$)
and consequently all its edges have grey tone qual to $lt_G$.
Therefore, statements 3 and 4 hold and the proof is finished.

\end{proof}

Now we can state the main results of this section.

\begin{thm}\label{thm:coulor}
Given a connected graph $G(V,E)$ with lightest tone $lt_G=\frac{1}{k}$ and $f$ a maximum contrast greyscale, 
then  a $(k+1)$-colouring of $G$ is obtained from $f$.
\end{thm}

\begin{proof}
Let us show that the following mapping  $\Phi: V \longrightarrow \{0, 1,
\dots, k \}$, is a $(k+1)-$colouring of $G$:

$$\Phi(v)=\left\{\begin{array}{ccc}
                    i & \mbox{if} \ \ \ \ \ \ \frac{i}{k}\leq f(v) < \frac{i+1}{k}  & \mbox{for}\  \ \ \ \ \  i=0,\dots ,k-1 \\
                     & & \\
                    k & \mbox{if} \ \ \ \ \ \  f(v)=1  &
                   \end{array}
                  \right.$$

Since Lemma \ref{pricamino} provides  one incremental path of length $k$
whose vertices  take the values $\frac{i}{k}$ for $i=0,1,\dots ,k$, the
chromatic  classes induced by $\Phi $ are not empty. Next we prove
that  any two vertices $u, v \in V$ with $\Phi (u)= \Phi (v)$ are
not adjacent in $G$. Suppose on the contrary that there is an edge
$e=\{u,v\} \in E$ such that $\Phi (u)= \Phi (v)=i$ with $i\neq k$.
Since $\frac{i}{k} \leq f(u) < \frac{i+1}{k}$ and $\frac{i}{k} \leq
f(v) < \frac{i+1}{k}$ then $\widehat{f} (e)=\mid f(u)-f(v)\mid <
\frac{1}{k} $ contradicting the hypothesis $lt_G=\frac{1}{k}$. In addition, if $\Phi (u)= \Phi (v)=k$
then $\widehat{f} (e)=\mid f(u)-f(v)\mid =0$ contradicting   that 
 $cont_{max}(G)$ has no  component equal to  0.

\end{proof}

\begin{thm}\label{th:chi-1}
The lightest tone of a connected graph $G(V,E)$ is 
$lt_G=\frac{1}{\chi(G)-1}$.
\end{thm}

\begin{proof}
The existence of a $(k+1)-$colouring of $G$ is ensured from Theorem \ref{thm:coulor}.
 Consequently, $k+1 \geq \chi(G)$.

In order to prove that $k+1 \leq \chi(G)$, let  $\Phi: V
\longrightarrow \{0, 1, \dots, r \}$ be  a colouring of $G$ with
$r+1=\chi(G)$ colours and we define $f: V \longrightarrow [0, 1]$ a
greyscale of $G$ as $f(v)=\frac{\Phi (v)}{r}$ for all $v \in V$.
Hence,
 $$\widehat{f} (e)= \mid f(u)-f(v) \mid = \left| \frac{\Phi
(u)}{r} - \frac{\Phi (v)}{r}\right| = \frac{1}{r} | \Phi (u) - \Phi (v) | \geq \frac{1}{r}$$ for
every $e=\{u,v\} \in E$. Therefore, by  definition,  the first coordinate of $C_f$, namely  $c$, satisfies  $\frac{1}{r}\leq c$. 
 Then  the lightest tone of the maximum contrast vector also  verifies this inequality, that is 
  $\frac{1}{r} \leq c \leq  \frac{1}{k}$ and we reach $k+1 \leq r+1=
\chi(G)$. 

\end{proof}

It is well known that the problem of computing the chromatic number is NP-hard \cite{gj-cigtnc-79}.
As a consequence of   Theorem \ref{th:chi-1} we can conclude this section with the following fact.

\begin{thm}\label{thm:NP}
The \textsc{macg} problem is NP-complete.
\end{thm}
\begin{flushright}
$\square$
\end{flushright}

\section{ Set of values of maximum contrast greyscales}\label{sec:values} 

In this section we give a procedure to
obtain the set of all possible values of any maximum contrast greyscale, and consequently, the components of the maximum contrast vector are determined.   
In accordance with Section \ref{sec:problem}, for any connected graph $G(V,E)$, 
it is known that  
the set $\{\frac{i}{k}: \, i= 0, 1, \dots, k\}$ with $k=\frac{1}{lt_G}$ is a subset of $Im(f)$ (see Lemma \ref{pricamino} and recall  $lt_G =\frac{1}{\chi(G)-1}$ by  Theorem \ref{th:chi-1}). 
 It is not difficult to find  
a graph for which the maximum contrast greyscale $f$ verifies $I_{k} \subsetneq Im(f)$. In  Figure~\ref{fig:veccontrast} (top left) a maximum contrast greyscale of the wheel is given: 
$f([0,1,2,3,4,5])=\left[1,0,\frac 12, 0, \frac 23, \frac
13\right]$, hence $I_3\subsetneq Im(f)$.

From the definition of pair of neighbours closest to a vertex and  statement 2 of   Lemma  \ref{pricamino}, we observe that for any  vertex $v$  with colour other than the extreme tones,  there exists a path of lenght at least 2  verifying that $v$ is an interior point of  that  path. Moreover, the vertices of such path are coloured with an increasing sequence of grey tones and all edges have the same grey tone.   We may  guess  that
  the set $Im(f)$ consists of some sets of values corresponding to the colours of such paths  and there are certain relations between those sets.  
   Next we study  this fact in a more deeply way as follows.

 We  start by  defining  some auxiliary notions  and presenting some technical results which are concerned with  properties 
 of discrete sets of numbers which may  be 
 the grey tones of some paths.  
 We emphasize that the analysis of these properties are made  independently  from the notion of greyscale of graphs. 
 
A sequence
$[y_0,y_1,\dots,y_r]\subset [0,1]$, with $r\ge 2$,  is said to be an \textit{$h$-step chain} 
of length $r$ in $[0,1]$ if  $y_i-y_{i-1}=h$  for $i=1,\dots,r$.
A number $y_i$ for $i\neq 0,r$ is named \textit{interior
point}  and  $y_0$ and $y_r$ are called \textit{extreme points} of the \textit{$h$-step chain}. 
Let us observe that any $h$-step chain is characterized by its extreme points $y_0$
and $y_r$ and its length  $r$, being $h=\frac {y_r-y_0}{r}$. 

A  set of numbers $F\subset [0,1]$ is said to be an
\textit{$h$-minimum-step-enchained set}  if $F$ verifies the
following assertions:

\begin{enumerate}
	\item[(1)]\label{cond:1} There is an $h$-step chain in $F$ whose extreme points are precisely $\{0,1\}$ and there is no  other $q$-step chain in $F$ 
	with extreme points  $\{0,1\}$ and $q< h$.
	\item[(2)]\label{cond:2} For every  $y\in F-\{0,1\}$ there exists  a $p$-step chain $P$ in $F$ with extreme points $y_1$ and $y_2$ and
	 such that  $y$ is an interior point in $P$ for some $p\geq h$.
	
	 If   $p>h$ and   $y_i\notin \{0, 1\}$, for $i=1$ or $2$,  then $y_i$  is an interior point of a $q_i$-step chain
	  with $h\leq q_i< p$.
\end{enumerate}
 
An  $h$-minimum-step-enchained set $F$ is  \textit{maximal} if
it is not a proper subset of another $h$-minimum-step-enchained set.

\vspace*{0.5cm}

Let us illustrate these  notions with an example.  $F=\{0, \frac{1}{4},  \frac{3}{8},\frac{1}{2},
\frac{5}{8}, \frac{3}{4}, 1\}$ is a $\frac{1}{4}$-minimum-step-enchained set containing the $\frac
12$-step chain $[0, \frac{1}{2}, 1]$, the $\frac 38$-step chains $[0, \frac{3}{8},
\frac{3}{4}]$ and $[\frac{1}{4}, \frac{5}{8}, 1 ]$, and the $\frac 14$-step chain $[0, \frac{1}{4},
\frac{1}{2},\frac{3}{4}, 1] $. $F$ also contains the $\frac 18$-step chain $[ \frac{1}{4},
\frac{3}{8},\frac{1}{2}, \frac{5}{8}, \frac{3}{4}] $.

\begin{lemma}\label{minimo}

Let  $F\subset [0,1]$  be an  $h$-minimum-step-enchained set. The following statements
hold:

\begin{enumerate}
 \item There exists $k\in N$, $k\ge 2$, such that $h=\frac 1k$.

 \item $I_k=\{ \frac ik \; : \; i=0,\dots, k\}\subset F$.

 \item $I_k$ is the only  $\frac 1k$-step chain in $F$.

\end{enumerate}

\end{lemma}

\begin{proof}

By definition, there exists an  $h$-step chain of
length
  $k$ with extreme points $\{0, 1\}$.
 Since   $k\ge 2$, $[0,1]$ is divided into  $k$ subintervals of
length  $h=\frac{1}{k}$. Hence, $I_k$ is the set of points in this
$h$-step chain. This leads to statements  1 and  2.

Let $C=[y_0,y_1,\dots,y_r]\subset F$ be an $\frac 1k$-step chain
and let us suppose
 $y_0\neq 0$. Then, the second part of  assertion  (2) (definition of $h$-minimum-step-enchained set)  provides another
 $q$-step chain  verifying  $\frac 1k\leq q<\frac 1k$, which is impossible. In an analogous way, we reach a contradiction by supposing
 $y_r\neq 1$.
\end{proof}

Next  we design a recursive procedure that gives a maximal $\frac
1{k}$-minimum-step-enchained set  for each  $k\geq 2$. This set will be denoted by  $F_{k}$, later we will prove that it is unique. The procedure
 starts with the set $\{0,1\}$ and
adds all possible values  $y\in (0, 1)$ that guarantee the
definition of $\frac 1k$-minimum-step-enchained set. More particularly, the method adds all
possible $p$-step chains, with $p\geq\frac 1k$, in such a way
that its interior and extreme points verify  assertion (2) for minimum step equal to  $\frac 1{k}$.

We define the auxiliary  mapping $S_{H,k}$  which will help us in the checking of assertion
(2)  during the procedure given below.

Let  $H\subset [0,1]$ be a finite set of numbers and $k\in N$ with
$k\ge 2$.  The function $S_{H,k}:H\longrightarrow [0,1]$ is defined
by:
$$S_{H,k}(y)=\left\{\begin{array}{l}
              \min\{p\ge \frac 1k\ \ : y \ \text{is an interior point for some $p$-step chain in $H$  }\\\text{\ \ \ \ \ verifing assertion (2)}\}\\
              0 \  \text{otherwise}
             \end{array}\right.$$

The following procedure starts  with the set  $H=\{0,1\}$ and adds
  possible values of  new $p$-step chains
$[y_0,\dots,y_r]$ with $p\geq \frac 1{k}$ recursively whenever
$\max\{S_{H,k}(y_0),S_{H,k}(y_r)\}<p$ is true. These new values are stored in the set $H$ and $S_{H,k}$  must be
actualized in each step. The algorithm ends when it is not possible the addition of a
new  interior point of a chain that  satisfies  assertion (2) of the definition of $h$-minimum-step-enchained set.

\vskip.3cm
\textsc{Procedure:} \textsc{maximal enchained set (mes)}

\textbf{Input:} A natural number $k\ge 2$.

\textbf{Output:} The maximal $\frac 1{k}$-minimum-step-enchained set, $F_{k}$.

\begin{enumerate}
 \item Initialize $H,\ New \leftarrow  \{0,1\}$

 \item Initialize $S_{H, k}(y)= 0$ for all $y\in H$

 \item  \textbf{While} $|New|>0$ \textbf{do}

  \begin{enumerate}

   \item Initialize $New \leftarrow  \emptyset$

   \item \textbf{For each} $\{y_1,y_2\}\subset H$  \textbf{do}

   \begin{enumerate}

      \item Initialize $r \leftarrow 2$

      \item Initialize $p \leftarrow \frac{|y_2-y_1|}r$

      \item  \textbf{While} $p\ge \frac 1{k}$ \textbf{do}

      \begin{enumerate}

           \item  \textbf{If} $p>\max\{S_{H,k}(y_1),S_{H,k}(y_2)\}$ \textbf{do}

           \hskip 0.65cm  I.  Make $C_r$ the $p$-step chain with extremes  $\{y_1,y_2\}$.

           \hskip 0.5cm  II. \textbf{For each} $y \in C_r-\{y_1,y_2\}$ \textbf{do}

              \hskip 1.05cm  $\alpha$. \textbf{If} $y \notin H$ \textbf{then} $S_{H,k}(y)=p$

              \hskip 1cm  $\beta$. \textbf{If} $S_{H,k}(y)> p$ \textbf{then} $S_{H,k}(y)=p$

           \hskip 0.42cm  III. Actualize $New \leftarrow (New\cup (C_r-H))$

           \item Actualize $r \leftarrow r+1$

           \item Actualize $p \leftarrow \frac{|y_2-y_1|}r$

        \end{enumerate}

   \end{enumerate}

   \item Actualize $H \leftarrow (H\cup New)$

   \end{enumerate}

 \item \textbf{return} $F_{k}=H$

\end{enumerate}

Since the procedure starts with  two fixed numbers, these are $\{0,1\}$, and each step is deterministic,  the uniqueness of its output is straightforwardly  deduced. Notice that this affirmation is true whenever the finiteness of the procedure is proven (Theorem \ref{th:finite}).  
By construction,  the actualized $H$ after the execution of  Step 3.(c) of the procedure is a $\frac 1k$-minimum-step-enchained set. 
The output of the algorithm is the maximal  $\frac 1k$-minimum-step-enchained set, $F_k$, due to  the fact that every pair of possible values for $H$ are revisited in Step 3.(b) until  no  new value can be adjoined to  $H$.  
 Moreover, $F_k$ is a rational set due to the fact that only rational numbers are adjoined to $C_r$ in Step 3.(b)iii.A.I  of  the algorithm.

The finiteness of $F_k$ (and consequently of the procedure)
deserves a  detailed analysis which starts with the following technical lemma.  Let us establish   some notation.
First of all,   $F_k$ can be  described as $F_k=\cup_{i=0}^{\infty} A_i$ where
$A_0=\{0, 1\}$ and, for $i\geq 1,$
$$A_i =\{y\in F_k -\cup_{j=0}^{i-1}A_j \, :\, y \,\mbox{verifies assertion (2) such that the extremes} \, y_1 \,  \mbox{and} \, y_2$$
$$ \mbox{of the existing $p$-step chain satisfy}\,  y_1, y_2\in \cup_{j=0}^{i-1}A_j\}.$$ 

 For the sake of simplicity, from now on we will denote  $F_k^i=\cup_{j=0}^i A_j$. 
 
\begin{lemma}\label{lem:symetry}
	The following statements hold for any $A_i$ with $i\geq 1$:
	\begin{enumerate}
		\item Every  number  $y\in A_i$  is an interior point for a $p$-step chain in $F_k$ with extreme points $y_1<y_2$ and $\{y_1, y_2\}\cap A_{i-1}\neq \emptyset$, where $p=S_{F_k^i,k}(y)$.
		
		\item Symetry property: if $y\in A_i$ then $1-y\in A_i$ and $S_{F_k^{i},k}(y)=S_{F_k^{i},k}(1-y)$.
		
		\item  $A_i$ is a finite set.
		
		\item  Let be $p_i= \min \{S_{F_k^{i},k}(y),\, :\;  y\in A_i\}$, then $p_i>p_{i-1}$, where $p_0=0$.

	\end{enumerate}
	
\end{lemma}

\begin{proof}
	
	Statement  1 holds by definition of the set $A_i$.
	Statement 2 is demonstrated by induction over $i$:  	
	 it is  trivially  true for $A_0=\{0,1\}$ and let us suppose that it is also true for $A_j$ with $j\leq i$. Let us consider  $y\in A_{i+1}$. Then, $y$ is an interior point of a $p$-step chain with extreme points $y_1<y_2$, such that  $y_s\in F_k^i$  and
	$S_{F_k^{i+1},k}(y)>\max \{(S_{F_k^i,k}(y_1), S_{F_k^i,k}(y_2)\}$ for $s=1,2$. From the induction hypothesis, $1-y_s\in F_k^i$ and
	$S_{F_k^i,k}(1-y_s)=S_{F_k^i,k}(y_s)$, for $s=1, 2$. We make the $p$-step chain with extreme points $1-y_2<1-y_1$. Then, $1-y$ is an  interior point and $S_{F_k^{i+1},k}(1-y)=S_{F_k^{i+1},k}(y)>\max \{S_{F_k^i,k}(y_1), S_{F_k^i,k}(y_2)\}=\max \{S_{F_k^i,k}(1-y_1), S_{F_k^i,k}(1-y_2)\}$, therefore  $1-y\in A_{i+1}$.
	
	Statement 3 is demonstrated by induction over $i$: it is true for $i=0$, since $|A_0|=2$. Let us  suppose $|A_j|<+\infty$ for $j\leq i$, then $|F_k^i|=|\cup_{j=0}^{i}A_j|\leq |A_0|+|A_1|+\dots +|A_i|<+\infty$. According to  Step 3 (b) iii.A.I. of {\sc mes} procedure, in $A_{i+1}$ there are at most 	
	$\binom{|F_k^i|}{2}$ chains with length $2,3,\dots k$. 
	Therefore, $|A_{i+1}|$ has at most
	$$(1+ 2+\dots+(k-1))\binom{|F_k^i|}{2}= \frac{(k)(k-1)}2 \binom{|F_k^i|}{2}$$
	elements, that is, $|A_{i+1}|<+\infty$.
	
	For the proof of  statement 4, let us observe that if $y\in A_i$ then there is some $h$-step chain with extreme points $y_1<y_2$ and $h=S_{F_k^i,k}(y)>\max \{S_{F_k^{i-1},k}(y_1), S_{F_k^{i-1},k}(y_2)\}$ and from statement 2, $y_1$ or $y_2$ belongs to $A_{i-1}$, hence $S_{F_k^{i-1},k}(y_s)\geq p_{i-1}$ for some $s=1,2$. Then,
	$S_{F_k^i,k}(y)>\max \{S_{F_k^{i-1},k}(y_1), S_{F_k^{i-1},k}(y_2)\}\geq p_{i-1}$, that is,
	$S_{F_k^i,k}(y)>p_{i-1}$ for each $y\in A_i$. Since $|A_i|<+\infty$, it must be $p_i=\min \{S_{F_k^i,k}(y),\, \mbox{such that}\; y\in A_i\} >p_{i-1} $.
\end{proof}

\begin{thm}\label{th:finite}
	$F_k$ is a finite set for every $k\geq 2$.
\end{thm}

\begin{proof}
	By reductio ad absurdum, let us  suppose $|F_k|=+\infty$. Since  $F_k=\cup_{i=0}^{\infty} A_i$ and each set $A_i$ is  finite, then it must be $A_i\neq \emptyset$ for all $i$, and hence, $\{p_i\}_{i=0}^{+\infty}$ is an infinite strictly increasing succession of numbers with trivial upper bound $p_i<\frac 12$. Then, there exists $\lim_{i\rightarrow +\infty}p_i=p$.

	On the other hand, for each $a\in A_i$ and $0<a<\frac 13$, by the symetry property, the number $b=1-a>\frac 23$ lies in $A_i$  and $c=\frac{b}2=\frac{1-a}2 \in A_{i+1}$ because $c$ is an interior point of the $c$-step chain $[0,c,b]$ with $S_{F_k,k}(c)=c>\frac 13>S_{F_k,k}(b)=S_{F_k,k}(a)$. Let us observe that  $(c-\frac 13)=\frac{1-a}2-\frac 13=\frac{\frac 13-a}{2}$, that is, for every $a\in A_i$ with $a<\frac 13$ there exists  $c\in A_{i+1}$ with  $|c -\frac 13|=\frac 12 |a -\frac 13|$. Under the assumption $F_k$ is an infinite set, then value $\frac 13$ is an accumulation point in $F_k$ and then, $p=\lim_{i\rightarrow +\infty}p_i=\frac 13$.
	
	Now, by definition of limit, for all $\epsilon>0$, there is some $m$ such that if $i\geq m$ then $\frac 13-p_i<\epsilon$. Let be $y\in A_{m+1}$, then $y$ is an interior point of an $h$-step chain with extreme points $y_1<y_2$ and some $y_s\in A_m$ for $s=1,2$, where $h\geq p_{m+1}$. Let us suppose $y_2\in A_m$, then $y_2$ is an interior point of a $q$-step chain with extreme points $y_3<y_4$, where $q\geq p_{m}.$  Then $0<y_1<y<y_2<y_4<1$ and $y_4-y_1=(y_4-y_2)+(y_2-y)+(y-y_1)\geq p_{m+1}+p_{m+1}+p_m>3(\frac13-\epsilon)=1-3\epsilon$. recall that  $y_4-y_1\geq 1-\frac 1k$, and then, 
	$1-3\epsilon<1-\frac 1k$, so $\epsilon>\frac{1}{3k}$ which is a contradiction. Therefore, the assumption is false and the set  $F_k$ is  finite for every $k\geq 2$.
\end{proof}

As a consequence of Theorem \ref{th:finite}, if $H_i$ denotes the set $H$ after $i$ loops in {\sc mes} procedure, then the number of loops is finite because $F_k$ is finite. This demonstrate the finiteness of the procedure. On the other hand, from the proof  of statement  3 in Lemma~\ref{lem:symetry}, let us remark that $H_i$ can be computed in $\mathcal{O}(|H_ i|)$ time
complexity. This leads to an exponential time complexity in the worst case for large values of $k$. Nevertheless, it is worth to run the \textsc{mes} procedure for $k$ from 2 to 7. 
This  will allow us to obtain some examples of  maximum contrast of graphs as it will be shown below.  The following sets are
obtained: \vskip.3cm

$F_2=\{0,\frac 12,1\}$
\vskip.3cm

$F_3=\{0,\frac 13,\frac 12, \frac23,1\}$
\vskip.3cm

$F_4=\{ 0\, \frac{1}{4}, \frac{1}{3}, \frac{3}{8},\frac{1}{2}, \frac{5}{8}, \frac{2}{3},\frac{3}{4}, 1\}$
\vskip.3cm

$F_5= \{0, \frac{1}{5}, \frac{1}{4}, \frac{4}{15},
\frac{3}{10}, \frac{1}{3}, \frac{7}{20},
\frac{11}{30}, \frac{3}{8}, \frac{2}{5},
\frac{7}{15}, \frac{19}{40}, \frac{1}{2},
\frac{21}{40}, \frac{8}{15}, \frac{3}{5},
\frac{5}{8}, \frac{19}{30}, \frac{13}{20},
\frac{2}{3}, \frac{7}{10}, \frac{11}{15},
\frac{3}{4}, \frac{4}{5}, 1\}$

\vskip.3cm

$F_6= \{0, \frac{1}{6}, \frac{1}{5}, \frac{5}{24}, \frac{2}{9},
\frac{1}{4}, \frac{7}{27}, \frac{19}{72}, \frac{4}{15},
\frac{5}{18},\dots \}$
\vskip.3cm

$F_7=\{0, \frac{1}{7}, \frac{1}{6}, \frac{6}{35}, \frac{5}{28},
\frac{4}{21}, \frac{1}{5}, \frac{17}{84}, \frac{23}{112},
\frac{13}{63},\dots \}$
\vskip.3cm

The sets $F_6$ ad $F_7$ are collected in a data file which is available at \cite{cas-ga-ri-ro-vi-16}.  Table \ref{tab:cardinal}  shows the increase of  $F_k$ cardinality. Note that  the  cases for  $k\ge 6$ are particularly significant.
  \vskip.3cm

\begin{table}[h]
\begin{center}
\begin{tabular}{|l || c | c | c | c | c | c |}
\hline
k & 2 & 3 & 4 & 5 & 6 & 7  \\
\hline
$|F_k|$ & 3 & 5 & 9 & 25 & 145 & 19027\\
\hline
\end{tabular}\end{center}
\caption{Cardinality  of some $F_k$.}\label{tab:cardinal}

\end{table}

With these technical results we have 
established some properties of discrete sets of numbers that will be used to demonstrate the main theorem of this section.
 The next theorem gives a fundamental property of the set
$Im(f)$ for any  maximum contrast greyscale $f$ of a given graph $G$ with known chromatic number
$\chi (G)$.

\begin{thm}\label{values}

 Let $G(V,E)$ be a connected graph with chromatic number $\chi (G)=k+1$. For any  maximum contrast greyscale $f$ of
 $G$, $Im(f)$ is a $\frac 1k$-minimum-step-enchained
 set and $\ I_k\subset Im(f)\subset F_k$  is verified.
\end{thm}
 \begin{proof}
Let us consider an arbitrary maximum contrast greyscale $f$ of
$G$.  Firstly, we   show that  $Im(f)$ satisfies  assertion (1)  in the  definition of  $\frac 1k$-minimum-step-enchained
set. From Lemma \ref{pricamino} and Theorem \ref{th:chi-1}  it holds that 
 $I_k \subset Im(f)$   is a $\frac 1k$-step chain with extreme points $\{0, 1\}$, where $k=\chi (G)-1$.
  Suppose there exists a $q$-step chain with extreme points $\{0, 1\}$ and $q< \frac 1k$, namely $[0, q, ..., 1]$. Then, $q=f(v)$ for some $v \in V$ and, by Lemma \ref{previo}, there exists $u_1 \in N(v)$, a neighbour closest to $v$ on the left. Hence, $f(u_1)<f(v)$ and $\widehat{f}(\{u_1,v\})\leq q < \frac 1k$, the lightest tone, which is a contradiction. Therefore, $I_k$ is the only $\frac 1k$-chain  with extremes $\{0,1\}$ and  assertion (1) holds.

  In order to prove assertion (2) in the  definition of  $\frac 1k$-minimum-step-enchained
 set we use the following auxiliary mapping  $C:V\longrightarrow [0,1]$ defined as follows:

  $$C(v)=\min_{w\in f^{-1}(f(v))} \{ \widehat{f}(\{u,w\}) \, : \;  u\in N(w) \} $$ 
  In other words, $C$ computes the minimum grey tone over the set of all edges incident in  vertices coloured with the  grey tone $f(v)$.

Given a value $y\in Im(f)-\{0,1\}$, let   $v$ be a vertex  such that
$y=f(v)$.  Let us  consider $C(v)$ and a vertex $w\in C^{-1}(C(v))$.
 Since  $f(w)=y\not\in \{0,1\}$, the hypotheses of Lemma \ref{previo} 
 are verified and hence there is a pair of neighbours closest to $w$, namely $u_1$ and $v_1$
such that  $f(u_1)<f(w)<f(v_1)$ and $\widehat{f}(\{u_1,w\})=\widehat{f}(\{v_1,w\})=p$, where $p=C(v)$.
 Therefore,  $[f(u_1),f(w)=y,f(v_1)]$ is a  $p$-step chain. Moreover, and since $w \in N(u_1)$, it is held that $C(u_1)\leq \widehat{f}(\{u_1,w\}) =C(w)=C(v)=p$  (see Lemma \ref{previo}); 
 analogously $C(v_1)\leq p$.

 Let us suppose  $C(u_1)=C(w)=p$ and $f(u_1)\neq 0$. A similar reasoning gives rise to   $w_2 \in C^{-1}(C(u_1))$ and there exists $u_2$, the neighbour  closest to $w_2$ on the left  
 such that
  $f(u_2)=f(u_1)-p$, and the above $p$-step chain is enlarged on the left as follows  $[f(u_2),f(u_1),f(w)=y,f(v_1)]$. This procedure can be repeated  $r_1$ times until  $C(u_{r_1})<C(w)=p$ or else  $f(u_{r_1})=0$. This way, a left extreme point
 for the $p$-step chain with interior  point $y$ is found. In summary, we obtain 
 $ [f(u_{r_1}),\hdots, f(u_2),f(u_1),f(w)=y,f(v_1)]$.
The case  $C(v_1)=C(w)=p$ and $f(v_1)\neq 1$ can be tackled
similarly and the   $p$-step chain is completed on the right in the form:
 $[f(u_{r_1}),\hdots, f(u_2),f(u_1),f(w)=y,f(v_1),f(v_2),\hdots,f(v_{r_2})]$ whose extreme points are 0 or 1 or otherwise $C$ reaches a value less than
 $p$.

Notice that, from the definition of lightest tone of $G$,  for every vertex $u$,   $C(u)\ge \frac 1k$ holds, in particular $C(v)=p\geq \frac 1k$.
Therefore,   every $y\in Im(f)-\{0,1\}$ is an interior point of a
$p$-step chain with $p\ge \frac 1k$ and whose extreme points are $f(u_{r_1})=0$ or else $C(f(u_{r_1}))=q_1<p$ or 
$f(v_{r_2})= 1$ or else  $C(f(v_{r_2}))=q_2<p$. 
In such a case, a similar construction  gives rise to  a $q_i$-step chain with   $q_i<p$ and $f(u_{r_1})$ or $f(v_{r_2})$ are the corresponding interior points.

 We conclude that  $Im(f)$ is  a $\frac 1k$-minimum-step-enchained
 set and from the maximality  of  $F_k$ it is reached that $Im(f)\subset F_k$ and the proof is finished.
\end{proof}
 
 As an  immediate consequence of Theorems \ref{th:finite} and \ref{values} we  reach the next corollary which assures that the notion  of {\it  maximum contrast vector}  is well defined.

\begin{cor}
	For any  graph $G$, the set $$\{cont(G,f)\, \mbox{such that} \,  f  \; \mbox{is a greyscale of }\, G\}$$  has a maximum.
	\end{cor}

  A natural question arises: we wonder if the bound set $Im(f)\subset
F_k$ for a given maximum contrast greyscale $f$ is tight or otherwise
there exists  $y\in F_k$ such that no graph $G(V,E)$ with $\chi
(G)=k+1$ and no maximum contrast greyscale $f$ of $G$  verifies
$y\in Im(f)$. For $k=2$ and the complete graph $G=K_3$, $Im(f)=F_2$. 
Figure \ref{fig:veccontrast} ((a) top) and Figure \ref{fig:K3-W6-G12V} show two graphs for which  $Im(f)=F_k$, with
$k=3,4$, respectively. A brute-force algorithm has been implemented for
these two graphs in order to check that a maximum contrast
greyscale has precisely these sets of values. Let us remark that in the latter case, the computation for  $k=4$ requires at
most $12^9$ combinations of values to test maximum contrast
greyscales. In the general case, we guess that the bound
$Im(f)\subset F_k$ is tight but constructing particular examples
verifying it is a difficult task since an amount of vertices  $|V|\ge|F_k|$ is required and
we are dealing with an NP-complete problem. \vskip 1cm

\begin{figure}[htb]

\begin{center}\includegraphics[width=6cm]{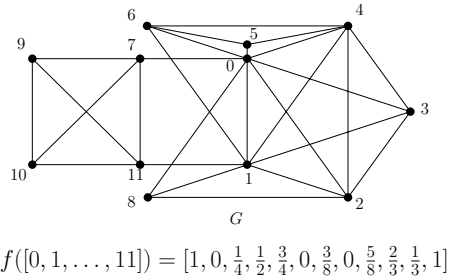} \caption{A maximum contrast greyscale 
 of the graph $G$ with $\chi(G)=5$ and  $Im(f)= F_4$. }
\label{fig:K3-W6-G12V}
\end{center}
\end{figure}

\section{Restricted maximum contrast problem on bipartite graphs}\label{sec:restricted}

In this section we consider the family of
bipartite graphs or equivalently 2-chromatic graphs. 
For these graphs we deal with the version of the maximum  contrast
problem in which the graph has several vertices initially coloured with
extreme tones 0 (white) and 1 (black)
and the rest of the vertices must be coloured by preserving those
initial colours. We name it  \textbf{$\{0,1\}$-Restricted maximum contrast on graphs } ($\{0,1\}$-\textsc{rmacg}) and it was
formally defined in Section \ref{sec:prelimi}. 

\par

Recall that $V_c$ denotes the set of vertices initially coloured, in particular, $V_i$ denotes the set of vertices initially coloured with extreme tone $i$, for $i=0, 1$. All over this section we consider that no pair of vertices in $V_i$ are adjacent, for $i=0,1$. 
Let us observe that in case that adjacent vertices are allowed to  be initially coloured with the same extreme tone, then the first component of the maximum contrast vector is 
 equal to zero. In this particular case, the solution of  $\{0,1\}$-\textsc{rmacg} problem is essentially the same as in the case in which the edges with grey tone equal to zero are removed from the graph. Hence it is worth to focus on the problem  with this extra condition of no adjacent vertices having the same initial colour.

 It suffices to revisit the proofs of  Lemma \ref{previo} and Proposition \ref{previo2} to realize that  similar results hold for the  $\{0,1\}$-\textsc{rmacg} problem, which are included next.
 
\begin{lemma}\label{lem:previo-restricted}
	Let  $G(V,E)$ be a connected graph and
	let   $f$ be a maximum contrast greyscale compatible with
	an incomplete $V_c$-greyscale $g$.
	Let $v\in V - V_c$ be a vertex  such that   $0<f(v)<
	1$, then there exist  $u_1$ and $u_2\in N(v)$  satisfying both
	following assertions:
	\begin{enumerate}
		\item $f(u_1)<f(v)<f(u_2).$
		\item $\widehat{f}(\{u_1,v\}) = \widehat{f}(\{u_2,v\}) =min \{\widehat{f}(\{u,v\}) : u\in N(v)\}.$
	\end{enumerate}
\end{lemma}

\begin{prop}\label{Ernuevo}
	Let $G(V,E)$ be a connected graph and 	let   $f$ be a maximum contrast greyscale compatible with
	an incomplete $V_c$-greyscale $g$.  If there exists a vertex $v\in V$ such that $0<f(v)<1$, then  the first component of $\mathcal{C}_f$ is at most  $\frac12$. Moreover, if the first component of  $\mathcal{C}_f$ is equal to $\frac12$, then  $Im(f) =  \{0,\,  \frac{1}{2}, \, 1\}$.
\end{prop}
	
The pair of vertices $u_1$ and  $u_2$ associated to  a vertex $v$, given by Lemma \ref{lem:previo-restricted}, will be referred as the  pair of neighbours closest to
		$v$, the vertex $u_1$ is the neighbour closest to $v$ on
		the left and the vertex $u_2$ is  the neighbour closest
		to $v$ on the right.
	
From here on, unless other thing is stated, we consider the set   $V_c$  having all initially coloured vertices with extreme tones. 
Let us fix some  notation that will help in the remaining of this paper. Any 2-chromatic graph $G$ has precisely two 2-colourings, say $\phi_0$ and $\phi_1$, both using colours 0 and 1. Let $g$ be
an incomplete $V_c$-greyscale of $G$.
Let us partition the set
of initially coloured vertices as follows:  $V_c=V_{\phi_0}\cup V_{\phi_1}$, where $V_{\phi_0}$ denotes the
set of vertices of $V_c$ whose colour coincides with the colour assigned by the colouring $\phi_0$.
Analogously, $V_{\phi_1}$ denotes the set of vertices of $V_c$ whose colour coincides with the colour
assigned by the colouring $\phi_1$, and hence their colours are opposite to the colour assigned by
$\phi_0$. Let $f$ be a maximum  contrast greyscale  for the  $\{0,1\}$-\textsc{rmacg} problem.  Observe that
$V_c=V_{\phi_0}$ (equivalently $V_c=V_{\phi_1}$) implies $Im(f)=\{0, 1\}$, that is, $f=\phi_0$ ($f=\phi_1$,
respectively).   The case $V_{\phi_0}\neq \emptyset$ and $ V_{\phi_1}\neq
\emptyset$   will be also  studied.
In fact, both situations occur for certain bipartite graphs.

Next, we  give an upper bound set for $Im(f)$ being $f$  a maximum  contrast greyscale for the  $\{0,1\}$-\textsc{rmacg} problem on the  family of 2-chromatic graphs and the set
$V_c$ of arbitrary cardinality.

\begin{thm}\label{thm:restrictedvalues} Let $G(V, E)$ be a  2-chromatic connected graph and let   $f$ be a maximum contrast greyscale for the $\{0,1\}$-\textsc{rmacg} problem  on $G$. Then,   it is  verified that $Im(f)\subseteq
\{0,\, \frac{1}{3},\, \frac{1}{2}, \, \frac{2}{3}, \, 1\}$.

\end{thm}

\begin{proof}
Let us consider  the set of 
initially coloured vertices $V_c\subset V$ with  incomplete
greyscale $g$. Consider the 2-colouring   of $G$, $\varphi=\phi_i :V \rightarrow \{0,1\}$, $i=0, 1$, such that $V_{\phi_i} \neq \emptyset$ (it is not difficult to  check that at least one of the two 2-colourings $\phi_1$ or $\phi_2$ of $G$ satisfies this condition).
 We define the following greyscale on $G$:

$$ f_{\varphi}(v)=\left\{ \begin{array}{lcl}
             g(v)& if & v\in V_c\\
             \frac{2}{3} & if & v\notin V_c \,\mbox{and there is}\, u\in V_c\cap N(v) \, \mbox{and}\, \varphi(u) \neq
             g(u)=0\\
             \frac{1}{3} & if & v\notin V_c \, \mbox{and there is}\, u\in V_c \cap N(v) \, \mbox{and}\, \varphi(u) \neq
             g(u)=1 \,  \\
             & \mbox{and} &\mbox{there is no}\, u\in V_c \cap N(v) \, \mbox{such that}\,  \varphi(u) \neq
             g(u)=0\\
             \varphi(v)& & \mbox{otherwise.}
\end{array}
\right.$$

It is readily checked that  $f_{\varphi}$  is well defined  and  compatible with $g$. Besides, since $G$
is bipartite  and equivalently  it has no  odd cycles, 
$Im(\widehat{f_{\varphi}})\subseteq \{\frac{1}{3},\, \frac{2}{3},\, 1\}$ holds.  Hence, the contrast vector $\mathcal{C}_{f_{\varphi}}$ has first component  $\frac{1}{3}$ and, therefore, the maximum  contrast vector compatible with $g$  must have first component  $a\geq \frac{1}{3}$.

Next let $f$ be a  maximum contrast  greyscale compatible with  $g$ and
let us suppose  that there is a vertex $u\in V$ such that $f(u)\notin \{0,\, \frac{1}{3},\,
\frac{1}{2}, \, \frac{2}{3}, \, 1\}$. From Lemma  \ref{lem:previo-restricted}, there exists the pair of neighbours
closest to $u$, $u_1$ and $u_2$, such that $f(u_1) < f(u) < f(u_2)$ and $f(u)-f(u_1)=f(u_2)-f(u)=d$.
 We analyse the value $d$. Observe that
$d< \frac{1}{2}$, since  $d=\frac{1}{2}$ implies $f(u)=\frac{1}{2}$ which is ruled out. On
the other hand, $\widehat{f}(\{u,u_1\})=d\geq a\geq \frac{1}{3}$, thus $\frac{1}{3}\leq d <
\frac{1}{2}.$

Since $d <\frac{1}{2}$, it is clear that $f(u_1)\neq 0$ or $f(u_2)\neq 1$. Let us  suppose
$f(u_1)\neq 0$ and consider the  pair of neighbours closest to $u_1$, namely $u_3$ and $u_4$, given by
Lemma \ref{lem:previo-restricted}. Then, $f(u_3) < f(u_1) < f(u_4)$ and $\frac{1}{3}\leq
f(u_1)-f(u_3)=f(u_4)-f(u_1)\leq d$.
Therefore $f(u_2)-f(u_3)\geq \frac{1}{3} +2d\geq 1$, which implies $d= \frac{1}{3}$ and $f(u)=
\frac{2}{3}$ contradicting our assumption. In a similar manner, if $f(u_2)\neq 1$ we reach
$f(u)=\frac{1}{3}$ which is also impossible. 

Then, we conclude $Im(f)\subseteq \{0,\, \frac{1}{3},\,
\frac{1}{2}, \, \frac{2}{3}, \, 1\}$  and the proof is finished.

\end{proof}

The bound set given by  Theorem \ref{thm:restrictedvalues} is tight  in the sense that there
are examples of graphs for which the full  set $\{0,\, \frac{1}{3},\, \frac{1}{2}, \, \frac{2}{3},
\, 1\}$ is needed. Figure \ref{optimalset} shows two of such examples.

\begin{figure}

\begin{center}\includegraphics[width=9cm]{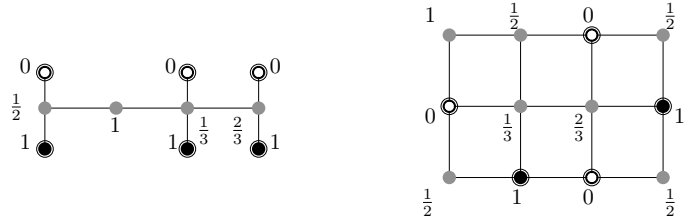}\end{center}

\caption{ Maximum contrast greyscales for the $\{0,1\}$-\textsc{rmacg}
problem on this tree and this grid are given. Double circles in
vertices denote the set $V_c$ of initially coloured vertices.} \label{optimalset}

 \end{figure}

 However, next theorem gives a characterization of $Im(f)$ on the family  of complete bipartite graphs, where $f$ never takes the  values $\frac{1}{3}$ and $\frac{2}{3}$ no  matter the cardinality  of $V_c$ is.

\begin{thm} \label{thm:Knm}
	Let $f$ be a maximum contrast greyscale for the $\{0, 1\}$-\textsc{rmacg} problem on any complete  bipartite graph. Then, its  maximum  contrast vector  has either all components equal to 1 or else all components equal to $\frac 12$. 
\end{thm}

\begin{proof} 
	 Let   $f$ be a maximum contrast greyscale compatible with
	 an incomplete $V_c$-greyscale $g$ on the complete  bipartite graph $K_{r,s}$.  Since  $V_c$ has no  adjacent vertices with the same initial extreme tone, only two possibilities  occur. The first one is that all vertices of $V_c$ have assigned 
	the colour of either $\phi_0$ or  $\phi_1$,  that is, $V_c=V_{\phi_0}$  or $V_c= V_{\phi_1}$, respectively. Without loss of generality, we consider $V_c=V_{\phi_0}$.  
  In this case, precisely  $f=\phi_0$ 
  gives an affirmative answer to   the $\{0, 1\}$-\textsc{rmacg} problem on $K_{r,s}$. Therefore, its maximum contrast vector   has all the components equal to 1. The other possibility is produced when $V_{\phi_0}\neq \emptyset$ and $ V_{\phi_1}\neq \emptyset$.  
	It is deduced that  $V_c$ is
	a subset of one of the chromatic classes  of the complete
	bipartite graph $K_{r,s}$.  Thus it is straightforwardly checked that all vertices of the other chromatic class of the graph must have assigned the
	grey tone  $\frac{1}{2}$   in order to give a  maximum contrast greyscale  for the $\{0,1\}$-\textsc{rmacg}
	problem on $K_{r,s}$. Hence, its maximum contrast vector has all the components equal to  $\frac{1}{2}$.
	
\end{proof}

This result  shows that  the $\{0, 1\}$-\textsc{rmacg} problem is solved  on the class of complete bipartite graphs. The general  case of our problem on bipartite graphs is far from having a trivial answer. Next Theorems \ref{thm:three} and  \ref{contrast-arbol} illustrate results for particular situations. In the first one  only  one  vertex lies in $V_{\phi_0}$ or $V_{\phi_1},$  while in the second one the bipartite graph is a  tree.
Below we include the technical  Lemmas \ref{lem:restrictedvalues2} and \ref{arbol} needed for proving the announced theorems. We denote $P_{uv}$ a $u-v$ path in $G$, that is a path joining the vertices $u$ and  $v$ in $G$ as it is usually   collected in the literature (see \cite{h-gt-90}).   According to  the notation introduced above,  the following result is given.

\begin{lemma}\label{lem:restrictedvalues2} Let   $f$ be a maximum contrast greyscale for the $\{0, 1\}$-\textsc{rmacg} problem   on a 2-chromatic  connected graph $G$. 
Then, for any
pair of vertices $u\in V_{\phi_0}$, $v \in V_{\phi_1}$ and for any $u-v$ path in $G$ 
there is a vertex $w$ in this path such that  $f(w)\notin\{0,1\}$. 
\end{lemma}
\begin{proof} 

Let   $P_{uv}$ be a
$u-v$ path in $G$ with length $l$. Since $G$ is a bipartite
graph, for any  2-colouring $\phi_i$ of $G$ it is verified that either $l$ is an
even number if and only if $\phi_i(u)=\phi_i(v)$ or else $l$ is an odd
number if and only if $\phi_i(u)\neq \phi_i(v)$, for $i=0,1$. Therefore, for the greyscale
$f$ there are  only two possibilities: either $f(u)=f(v)$ and $l$ is
odd or else $f(u)\neq f(v)$ and $l$ is even. It is
straightforwardly checked that there is a vertex $w\in P_{uv}$ such that $f(w)\notin\{0,1\}$, since otherwise, the first component of the maximum contrast vector $\mathcal{C}_f$ is 0. 
\end{proof}

 Next result gives a bound set of $Im(f)$ for  the case of 2-chromatic connected graphs with $V_{\phi_0}$ or $V_{\phi_1}$
  having precisely  one vertex. 

\begin{thm} \label{thm:three}
 Let   $f$ be a maximum contrast greyscale for the $\{0, 1\}$-\textsc{rmacg} problem   on a 2-chromatic connected graph $G$, with $|V_{\phi_0}|=1$ or $|V_{\phi_1}|=1$. Then,   it is
verified $Im(f)\subseteq \{0,\,  \frac{1}{2}, \, 1\}$.
\end{thm}

\begin{proof}
Let   $f$ be a maximum contrast greyscale compatible with
an incomplete $V_c$-greyscale $g$ of $G$.  
Without loss of generality, we deal
with the case  $V_{\phi_0}\neq \emptyset$ and $ V_{\phi_1}\neq \emptyset$, since otherwise $Im(f)=\{0,
1\}$. Besides, we can suppose $V_{\phi_0}=\{v_0\}$
and the case  
$V_{\phi_1}=\{v_0\}$ is analogous.

From Lemma \ref{lem:restrictedvalues2}, $f$ takes at least a grey tone  different from $0$ and $1$, and from Proposition \ref{Ernuevo}, the first component of the  vector $\mathcal{C}_f$ is at most $\frac 12$.

Let us  define the following  greyscale:

$$f_0(v)=\left\{ \begin{array}{lcl}
g(v)& if & v\in V_c\\
\frac{1}{2} & if & v\in N(v_0)\\
\phi_1(v)& & \mbox{otherwise.}
\end{array}
\right.$$

It is   checked that $f_0$  is well defined. In particular, no pair of adjacent vertices are
coloured with the same grey tone since there are no triangles in the  subgraph induced by $N(v_0)$  and    $\phi_1$ is a proper colouring on  $V-N(v_0)$. Hence, the first component of the vector $\mathcal{C}_{f_0}$ is $\frac 12$. Therefore, since $\mathcal{C}_f\ge \mathcal{C}_{f_0}$, then $\mathcal{C}_f$ has its first component  equal to $\frac 12$. Finally, from Proposition \ref{Ernuevo} we reach $Im(f)= \{0,\, \frac{1}{2}, \, 1\}$.

\end{proof}

\begin{prop}\label{arbol}

Let $T$ be a subdivision of the star graph $K_{1,n}$ with  leaves $\{v_1,\,
\dots, \, v_n\}$, for $n\geq 3,$   $u$  the vertex of
degree $n$ in $T$ and $g$ an incomplete greyscale such that $g(v_i)\in
\{0, \, 1\}$, for $1\leq i\leq n$. 
For any  maximum contrast greyscale  
compatible with $g$, $f$  of  $T$, it is verified that 
$Im(f)\subseteq  \{0, \frac{1}{2}, 1\}$. Moreover,  $f$ uses precisely the grey tone $
\frac 12$  over   at most   $\lfloor \frac{n}{2}\rfloor$  vertices, each one  of them lying in a different  path
$P_{uv_i}$.

\end{prop}

\begin{proof}
  Firstly,  we observe that in  case that $T$ is equal  to $K_{1,n}$ and  $g(v_i)=0$  for $1\leq i\leq n$   or $g(v_i)=1$  for $1\leq i\leq n,$ the grey tone $\frac{1}{2}$ is not used in $f$, and hence   the result is held.
  Next, let us observe that  since the paths $P_{uv_i}$ intersect  only in the vertex $u$, we can define a greyscale $f'$ that assigns $f'(u)=\frac{1}{2}$, 
   $f'(v_i)=g(v_i)$  and extends this colouring for each $P_{uv_i}$ by starting at $v_i$ and alternates 0 and 1 until the vertices of $N(u)$  are reached. $\mathcal{C}_{f'}$ has precisely its $n$ first components 
    equal to  $\frac{1}{2}$ and the rest of components are equal to 1. Hence the first component of the maximum contrast vector 
    for the $\{0, 1\}$-\textsc{rmacg} problem   is greater than or equal to  $\frac{1}{2}$ and therefore, by applying Proposition \ref{Ernuevo}, any maximum contrast greyscale compatible with $g$, $f$, verifies  $Im(f)\subseteq \{ 0, \frac{1}{2}, 1\}$.
    
  Moreover, in case that  $f(u)=\frac{1}{2}$, then $\mathcal{C}_{f'} =\mathcal{C}_{f}$, hence $u$ is the only vertex with grey tone equal to $\frac{1}{2}$. 
   We reach to the assertion of the Lemma since $\lfloor \frac{n}{2}\rfloor\geq 1$.

Now, let us consider the case in which $f(u)\neq \frac{1}{2}$. Without loss of generality we set 
the colouring $\phi_0$  to be the one that assigns the colour 0
to the vertex $u$ and denote by $n_0$ the number of paths $P_{uv_i}$ verifying $g(v_i)=\phi_0(v_i)$.
 Next, the colouring $\phi_1$  assigns
the colour 1  to the vertex $u$ and let us  denote  by $n_1$ the number of
paths $P_{uv_i}$ verifying 
$g(v_i)=\phi_1(v_i)$.

Next,   we select the colour  $c$ for $u$ such that $n_c=max\{n_0,
n_1\}$. 
 By the classical pigeonhole principle, it is deduced that $n_c\geq \lceil \frac{n}{2}\rceil$. 

Finally, let us define $f''$ a greyscale compatible with $g$ in the following way:

 $$f''(v)=\left\{\begin{array}{cl}
 g(v) &  {\mbox if} \, v=v_i\,  \mbox{ for}\, i=1, \dots, n   \\
 \frac{1}{2} & \mbox{if}\, v\in N(v_i)\, \mbox{ for each vertex }\, v_i\, \mbox{ such that}\, \phi_c(v_i)\neq g(v_i)  \\
\phi_c(v) &  \mbox{in other case.} \\
 \end{array}
 \right.$$
  
This way,  only a vertex of $P_{uv_i}$ has  grey tone $\frac{1}{2},$ for  at most $n-\lceil
\frac{n}{2}\rceil=\lfloor \frac{n}{2}\rfloor$  paths $P_{uv_i}$. 
We conclude that since the  maximum  contrast greyscale compatible with $g$, $f$ of $G,$ verifies $\mathcal{C}_f\geq \mathcal{C}_{f''}$, the assertion holds.

\end{proof}

\vspace*{.5cm}

 The next result shows another case in which the grey tone $\frac{1}{2}$ plays an important role in the maximum  contrast greyscale for the $\{0, 1\}$-\textsc{rmacg} problem. It is tackled in the family  of trees  with precisely three initially coloured vertices and a method
to assign the value $\frac{1}{2}$ is described into the proof.

\begin{thm}\label{contrast-arbol}
Any maximum  contrast greyscale  $f$ for  the $\{0, 1\}$-\textsc{rmacg} problem on a tree $T=G(V, E)$ with an incomplete $V_c$-greyscale $g$ and $|V_c|=3$ verifies that
$|\{v\in V : f(v)=\frac{1}{2}\}|\leq 2.$
\end{thm}

\begin{proof}  
Let   $V_c=\{v_1,\, v_2, \, v_3\}\subset V$ and let us consider  the $2$-colourings $\phi_0$ and $\phi_1$ on $T$.  Since  $T$ is connected,  the set  $V$ is partitioned  into two chromatic classes in a unique way.  Let us select the $2$-colouring $\phi_i$ such that  the cardinality  of the set $V_{\phi_i}$ is greater than or equal to  2. Let us suppose $|V_{\phi_0}|\geq 2$ and the other case is analogous.

 If $|V_{\phi_0}|=3,$  the greyscale $f$ defined by $f(v)=\phi_0(v)$, for any $v\in V$, is a  maximum  contrast greyscale   for the $\{0, 1\}$-\textsc{rmacg} problem and there is no vertex with grey tone $\frac 12$, hence the result holds.
 
  Otherwise, $|V_{\phi_0}|=2$, hence $|V_{\phi_1}|=1$ and then,  from  Theorem \ref{thm:three}, we get   $Im(f)=\{0,\, \frac{1}{2}, \, 1\}$. Moreover, the number of components equal to  $\frac 12$ in the maximum contrast vector $\mathcal{C}_f$  coincides with the number of vertices adjacent to  those vertices coloured with the grey tone $\frac 12$, that  is  the sum of the degrees  of  all  vertices coloured with $\frac 12$. 
 Without loss of generality, let us suppose  $\phi_0(v_1)=g(v_1),$  $\phi_0(v_2)=g(v_2)$ and $\phi_0(v_3)\neq g(v_3).$ By Lemma \ref{lem:restrictedvalues2}, there is a vertex $w_1\in P_{v_1v_3}$ and a vertex $w_2\in P_{v_2v_3}$ with $f(w_1)=f(w_2)=\frac12$. 
Since  $f$ is a maximum contrast greyscale  compatible with $g$ on $T,$  the restriction of $f$ to the
union of the three paths $P_{v_iv_j},$ for $1\leq i<j\leq 3,$  assigns the grey tone $\frac 12$ only to $w_1$ and $w_2$.  
Due to the fact that $T$ has no cycles, it is straightforwardly checked that, by starting  from each $v_i \in V_c,$  $f$ assigns 0 and 1   appropriately to the remaining vertices of  $T-\{V_c\cup \{w_1,w_2\}\}$. Hence,  $\{w_1,w_2\}$ are the only vertices with grey tone $\frac 12$. 

More precisely, let $W$ be the set of vertices within $V(P_{v_1v_3}\cup P_{v_2v_3}\cup P_{v_1v_2})-V_c$ with minimum degree in  $T$. If there exists a vertex $w\in W$ such that $w\in V(P_{v_1v_3})\cap V(P_{v_2v_3}),$ then $w_1=w_2=w$ and $f$ assigns the grey tone $\frac 12$ to precisely one of such a vertex $w$.
Otherwise, without loss of generality, let us suppose $w\in V(P_{v_1v_3})-V(P_{v_2v_3})$ and consider a pair of vertices  $w'\in V(P_{v_1v_3})\cap V(P_{v_2v_3})$  and $w''\in V(P_{v_2v_3})- V(P_{v_1v_3})$ both with minimum degree in $T$. If $deg(w')\leq deg(w)+deg(w''),$ then   $w_1=w_2=w'$ and $f$ assigns the grey tone $\frac 12$ to precisely  one of such a vertex $w'$. If not,  $f$ assigns $\frac 12$ to precisely $w_1=w$ and $w_2=w''$ or to another pair of vertices $w$ and $w''$ verifying the same conditions as $w_1$ and $w_2,$ respectively.

\end{proof}

\vspace*{0.5cm}

We guess that finding the solution of the $\{0,1\}$-\textsc{rmacg} problem on
bipartite graphs  is NP-complete, in general. 
 In fact, we have given several examples showing that there are trees with  few vertices that  need precisely  
 the full set of 5 grey tones given by  Theorem \ref{thm:restrictedvalues} (see Figure \ref{optimalset}). Many other examples of trees  with similar properties  may be found.
Hence,  finding the
solution of the $\{0,1\}$-\textsc{rmacg} problem is not an easy question to
answer even for the case of trees.

\section{Results and open questions}
\label{sec:conclusions}

We have introduced the new concept of contrast of a graph related to vertex and edge colourings. The maximum contrast problem (\textsc{macg} problem) has been studied in the   general case and links with the chromatic number of a graph is presented in Theorems \ref{thm:coulor} and \ref{th:chi-1}. 
As a conclusion, we get that the proposed problem belongs to NP-complete problems category (Theorem \ref{thm:NP}). 

However, we have achieved several results that allow us  to compute the set of all possible values of grey tones for maximum contrast greyscales of graphs with known chromatic number. Moreover, some notions of subsets in $[0, 1]$ are introduced, such as  the $h$-minimum-step-enchained set. We want to remark here that this kind of sets may be considered independently from any graph and could be useful in other branches of Mathematics.  By using the algorithmic  procedure  \textsc{maximal enchained set} included in this paper we can prove that the set of possible values of grey tones of a maximum contrast greyscale of a graph is a rational finite set (Theorems \ref{th:finite} and \ref{values}).

Hence, some natural questions remain open:  to design  heuristics to  approximate the solution of the \textsc{macg} problem on graphs and for tackling this problem  in 
particular families of graphs in a more efficient way.

On the other hand, we give another  version  of the \textsc{macg} problem consisting of solving the same question but initially assigning some extreme tones to a subset of vertices. We  name it $\{0,1\}$-Restricted maximum contrast on graphs, $\{0,1\}$-\textsc{rmacg} problem in short. We have provided a bound set of values of a solution for the $\{0,1\}$-\textsc{rmacg} problem  on any  bipartite graph (Theorem \ref{thm:restrictedvalues}).  The particular case of the $\{0,1\}$-\textsc{rmacg} problem on bipartite complete graphs is  solved  (Theorem \ref{thm:Knm}). 
We guess that  the $\{0,1\}$-\textsc{rmacg} problem on any
bipartite graph  is NP-complete. Nevertheless,  we have solved this problem for bipartite graphs and for  particular family of trees   with  additional conditions on the set of vertices initially coloured (Theorems \ref{thm:three} and \ref{contrast-arbol}).
Hence, another improvement of these results like solving the problem  in other families of graphs such that outerplanar graphs or planar graphs would be dealt  in future works.
Besides, another open problem appears if we change  the restriction of initially coloured vertices with  extreme tones $\{0,1\}$ and consider any other grey tones for the incomplete greyscale. 

\vspace*{1cm}
\noindent{\bf Acknowledgements}

\vspace*{0.5cm}
The authors gratefully acknowledge financial support by the Spanish Ministerio de Econom\'{\i}a, Industria y Competitividad and Junta de Andaluc\'ia via grants,  MTM2015-65397-P (M.T. Villar-Li\~n\'an) and  PAI FQM-164, respectively.

%\bibitem{g-mdcia-98} R. P. Grimaldi, \emph{Matemáticas discreta y combinatoria. Una introducción con aplicaciones}.
%Addison Wesley Longman, Mexico (1998).

%\bibitem{fgmnrv-eqtpsp-10} I. Fern\'andez, C. I. Grima, A. M\'arquez, A. Nakamoto, R. Robles and J. Valenzuela.
%\emph{Even and quasi-even triangulations of point sets in the plane}. 26th European Workshop on
%Computational Geometry. Dortmund, Germany. 161-164 (2010).

%\bibitem{ps-evecop-86} A. Proskurowski and M. Syslo, Efficient vertex-and edge-coloring of outerplanar graphs,
%\emph{CIAM Journal on Algebraic and Discrete Methods}, Volumen 7(1), 131-136 (1986).

%\bibitem{co-aagt-93} G. Chartrand and O. R. Oellermann, \emph{Applied and Algorithmic Graph Theory}.
%McGraw-Hill, Singapure (1993).

%\bibitem{s-idmcgtm-90} S. Skiena, \emph{Implementing Discrete Mathematics: Combinatorics and Graph Theory with Mathematica}.
%Addison-Wesley (1990).
%\end{thebibliography}

\end{document}